\documentclass{amsart}
\usepackage{amsthm,amsmath,amssymb,latexsym,natbib,mathrsfs,mathtools}
\usepackage[implicit=false, bookmarks=false]{hyperref}
\theoremstyle{plain}
\newtheorem{thm}{Theorem}[section]
\newtheorem*{thm*}{Theorem}
\newtheorem*{bibl}{Bibliographical remark}
\newtheorem*{conv}{Convention}
\newtheorem*{quest}{Question}
\newtheorem{lem}[thm]{Lemma}

\newtheorem{cor}[thm]{Corollary}
\newtheorem{definition}[thm]{Definition}
\newtheorem{fact}[thm]{Fact}

\newcommand{\pa}{\mathrm{PA}}
\newcommand{\tq}{\mathrm{Q}}
\newcommand{\ID}{\mathrm{I}\Sigma_0 + \mathrm{exp}}
\newcommand{\ISn}{\mathrm{I}\Sigma_n}
\newcommand{\ISz}{\mathrm{I}\Sigma_0}
\newcommand{\ISo}{\mathrm{I}\Sigma_1}
\newcommand{\ISnp}{\mathrm{I}\Sigma_{n+1}}
\newcommand{\Th}{\mathrm{Th}}
\newcommand{\Sat}{\mathrm{Sat}}
\newcommand{\SSy}{\mathrm{SSy}}
\newcommand{\T}{\mathrm{T}}
\newcommand{\s}{\mathrm{S}}
\newcommand{\con}{\mathrm{Con}}
\newcommand{\tr}{\mathrm{Tr}}
\newcommand{\prf}{\mathrm{Prf}}
\newcommand{\pr}{\mathrm{Pr}}
\newcommand{\defeq}{\coloneqq}

\title[Hierarchical incompleteness]
      {Hierarchical incompleteness results for arithmetically definable extensions of fragments of arithmetic}

\author{Rasmus Blanck}
\address{Department of Philosophy, Linguistics and Theory of Science, University of Gothenburg, Box 200, SE-405 30, Gothenburg, Sweden}
\email{rasmus.blanck@gu.se}
\thanks{Accepted for publication in \emph{The Review of Symbolic Logic} under the terms of the Creative Commons Attribution license (\url{http://creativecommons.org/licenses/by/4.0}), which permits unrestricted re-use, distribution, and reproduction in any medium, provided the original work is properly cited.}


\begin{document}

\begin{abstract}
There has been a recent interest in hierarchical generalisations of 
classic incompleteness results. This paper provides evidence that 
such generalisations are readily obtainable from suitably formulated 
hierarchical versions of the principles used in the original proofs. 
By collecting such principles, we prove hierarchical versions of 
Mostowski's theorem on independent formulae, Kripke's theorem on 
flexible formulae, Woodin's theorem on the universal algorithm, 
and a few related results.
As a corollary, we obtain the expected result that the formula 
expressing ``$\T$ is $\Sigma_n$-ill'' is a canonical example of 
a $\Sigma_{n+1}$ formula that is $\Pi_{n+1}$-conservative over $\T$.
\end{abstract}

 \maketitle

\section{Introduction}
There has been a recent interest in hierarchical generalisations of
classic incompleteness results
\citep{chao_seraji2018, kikuchi_kurahashi2017, kurahashi2018, 
salehi_seraji2017}. A sample result, generalising the G\"odel-Rosser
incompleteness theorem, and independently proved by both 
\citet{kikuchi_kurahashi2017} and \citet{salehi_seraji2017}, is:

\begin{thm}\label{theorem:1}
 Let $\T$ be a $\Sigma_{n+1}$-definable, $\Sigma_n$-sound extension 
 of $\pa$. Then there is a $\Pi_{n+1}$ sentence that is undecidable 
 in $\T$.
\end{thm}

In this paper I argue that such hierarchical generalisations can 
often be obtained from the original proofs by replacing certain 
principles used in the proofs by appropriately formulated 
hierarchical versions, while the essence of the arguments 
remains the same. The hierarchical principles, once appropriately
formulated, are in turn often provable by appropriate generalisations
of the core concepts employed in the proofs of the ordinary ones, but 
even so, there is no single source to which to turn for them. Both 
\citet{smorynski1985} and \citet{beklemishev2005} give good partial 
accounts of the syntactical side, and \citet{poizat2000} gives a 
hierarchical perspective on the basic model theory of arithmetic, 
including model-theoretic proofs of hierarchical versions of 
G\"odel's first and second incompleteness theorems. Still, I find 
certain aspects lacking. With this in mind, one aim of this paper 
is to collect a number of principles that may be useful to the reader 
who herself wishes to prove hierarchical incompleteness results without 
having to reinvent the wheel.

These principles are then put to use to prove a number of general 
incompleteness results for arithmetically definable extensions of 
fragments of $\pa$. The goal is not to prove the sharpest or most 
general results (in fact some of the results follow from
each other), but rather to exemplify how the hierarchical 
principles enter into more or less well-known proof methods. Even so, 
the results presented here improve on some results of 
\citet{chao_seraji2018}, \citet{kikuchi_kurahashi2017}, and 
\citet{salehi_seraji2017}, and sharpen some of \citet{blanck2017a}, 
\citet{hamkins_20XXa}, \citet{lindstrom1984}, and \cite{woodin2011}. 
These sharpenings are in terms of gauging the amount of induction 
needed for the proofs, bringing the (in this particular sub-field 
largely ignored) fragments-of-arithmetic perspective to attention.

In order to state the results in a more general form, I have chosen 
to consider only extensions of $\mathrm{I}\Delta_0 + \mathrm{exp}$, 
although under the somewhat unusual name $\ID$.
This allows for formulating results for extensions of, e.g., 
$\ISn + \mathrm{exp}$, ensuring that partial satisfaction predicates 
are well behaved even for $n=0$. While it is sometimes possible to push 
the background theory below $\ID$, I have refrained from doing so, to 
instead focus on the more general hierarchical picture.

\section{Notation and conventions}

The expressions $\exists x \leq t \phi(x)$ and $\forall x \leq t \phi(x)$ 
are used as shorthand for $\exists x (x \leq t \land \phi(x))$ and 
$\forall x (x\leq t \rightarrow \phi(x))$, where $t$ is some term in
the language of arithmetic. The initial quantifiers of these formulae 
are \emph{bounded} and a formula containing only bounded quantifiers 
is a \emph{bounded formula}. Let $\Delta_0 = \Sigma_0 = \Pi_0$ be the
class of bounded formulae. 

The arithmetical hierarchy is defined as follows. A formula is 
$\Sigma_{n+1}$ iff it is of the form $\exists x_1 \dots x_m \pi(x_1,\dots,x_m)$ 
where $\pi$ is a $\Pi_n$ formula (that may contain other variables
than $x_1,\dots,x_m$). Similarly, a formula is $\Pi_{n+1}$ iff it is
of the form $\forall x_1 \dots x_m \sigma(x_1,\dots,x_m)$ where 
$\sigma$ is a $\Sigma_n$ formula.
$\Delta_n(\mathcal{M})$ ($\Delta_n(\T)$)
is the set of $\Sigma_n$ formulae that are equivalent to $\Pi_n$ 
formulae in a given model $\mathcal{M}$ (theory $\T$). 
Throughout the paper $\Gamma$ denotes either $\Sigma_{n+1}$ or 
$\Pi_{n+1}$, and we always assume only that $n\geq 0$. 

Theories are understood as sets of sentences, thought of as the set of 
nonlogical axioms of the theory. $\mathrm{I}\Sigma_n$ is the theory 
obtained by adding induction for $\Sigma_n$ formulae to Robinson's 
arithmetic $\tq$, while $\ID$ is $\tq$ plus $\Sigma_0$-induction plus 
an axiom stating that the exponentiation function is total. We assume 
that all theories denoted $\T$, etc., are consistent, arithmetically 
definable, extensions of $\ID$.

$\T$ is $\Gamma$\emph{-sound} iff for all $\Gamma$ sentences $\gamma$, if $\T \vdash \gamma$, 
then $\mathbb{N} \models \gamma$. The converse implication is sometimes known 
as $\Gamma$-completeness: hence $\T$ is $\Gamma$\emph{-complete} iff for all 
$\Gamma$ sentences $\gamma$, if $\mathbb{N} \models \gamma$, then $\T \vdash \gamma$.
$\s$ is $\Gamma$\emph{-conservative} over $\T$ iff for all $\Gamma$ sentences $\gamma$,
if $\s \vdash \gamma$, then $\T \vdash \gamma$.

We rely on a coding of finite sets and sequences in $\ID$, as developed by 
\citet[Chapter I.1]{hajek_pudlak1993}.
The set $\Sigma_0(X)$ is obtained by adding atomic formulae 
of the form $t \in X$ (where $t$ is any term) and closing under propositional
connectives and bounded quantifiers. The set $\Sigma_1(X)$ is obtained 
from $\Sigma_0(X)$ in the usual manner. 
 
 Let $\varepsilon$ be Ackermann's membership relation: $n\varepsilon a$ 
 expresses ``the $n$th bit of the binary expansion of $a$ is $1$''. 
 In other words, $a$ can be regarded as a code for the set consisting of
 all $n$ such that $n\varepsilon a$.
 Then $a$ is a \emph{$z$-piece of $\phi(x)$} if 
 $\forall n < z (n\varepsilon a \leftrightarrow \phi(n))$.

If $\phi(x)$ is any formula, $\ulcorner \phi(x) \urcorner$ denotes 
the numeral for the G\"odel number of $\phi(x)$ under some fixed 
G\"odel numbering, but we make no typographical distinction between 
natural numbers and the corresponding numerals. We use Feferman's 
dot notation $\ulcorner \phi(\dot{x})\urcorner$ to represent the 
G\"odel number of the sentence obtained by replacing the variable 
$x$ with the actual value of $x$; hence $x$ is free in $\ulcorner \phi(\dot{x})\urcorner$.
The notation $\defeq$ is used to express equality between formulae. 
Let $\top \defeq 0 = 0$ and $\bot \defeq \lnot \top$. Let $\phi^0 \defeq \lnot \phi$ 
and $\phi^1 \defeq \phi$. \footnote{The literature is not in total agreement 
about this convention. We follow \citet{hajek_pudlak1993}, 
while for example \citet{lindstrom2003} has it the other way around. The mnemonic
here is that the superscript $1$ signals that $\phi$ occurs positively.} 

A relation $X$ is \emph{numerated} in $\T$ by a formula $\phi$ iff 
$X = \{\langle k_1,\dots,k_n \rangle \in \omega^n : \T \vdash \phi(k_1,\dots,k_n)\}$,
and \emph{binumerated} by $\phi$ in $\T$ if $\lnot \phi$ also numerates 
the complement of $X$. A relation $X$ is \textit{correctly numerated} 
by $\phi$ if $\phi$ numerates $X$ and for all $k_1,\dots,k_n$, 
$\T \vdash \phi(k_1,\dots,k_n)$ iff $\phi(k_1,\dots,k_n)$ is true. 
A function $f$ is \emph{strongly representable} in $\T$ iff there is 
a formula $\phi(x_1,\dots,x_n,y)$ that numerates $f(x_1,\dots,x_n) = y$ 
in $\T$, and moreover, if $f(k_1,\dots,k_n) = m$, then 
$\T \vdash \forall y (\phi(k_1,\dots,k_n,y) \rightarrow y = m)$.

Given a formula $\tau(z)$, let $\prf_\tau(x,y)$ be a formula expressing 
``$y$ is a proof of the formula $x$ from the set of sentences satisfying 
$\tau(z)$''.
Let $\pr_\tau(x)$ be the formula $\exists y \prf_\tau(x,y)$, 
and let $\con_\tau$ be the sentence $\lnot \pr_\tau(\ulcorner \bot \urcorner)$. 
Whenever $\tau(z)$ is $\Sigma_{n+1}$,  $\pr_\tau(x)$ is equivalent to a 
$\Sigma_{n+1}$ formula in the real world. For any formula $\tau(x)$, the notation
$(\tau + z)(x)$ is used as shorthand for $\tau(x) \lor x = z$. This convention
is used in expressions such as $\prf_{\tau + z}(x)$.

Models of arithmetic are denoted $\mathcal{M}$, etc., while the respective
domains are denoted $M$, etc. The standard model is denoted $\mathbb{N}$ 
and its domain is $\omega$. If $\mathcal{M}$ is non-standard, then the \emph{standard
system} of $\mathcal{M}$, $\SSy(\mathcal{M})$, is the collection of sets 
$X \subseteq \omega$ such that for some $a \in M$, 
$X = \{ n \in \omega : \mathcal{M} \models n\varepsilon a\}$
Then $X$ is \emph{coded} in $\mathcal{M}$, and $a$ is a \emph{code for} $X$. 

A relation $X \subseteq M^n$ is $\Gamma$\emph{-definable in} $\mathcal{M}$ 
(with parameters) iff there is a tuple $\overline{b}\in M$ and a formula 
$\phi(x_1,\dots,x_n, \overline{y}) \in \Gamma$ such that 
$X = \{\langle m_1,\dots,m_n \rangle \in M^n : \mathcal{M} \models \phi(m_1,\dots,m_n,\overline{b})\}$.
Whenever this terminology is used without specifying a model $\mathcal{M}$, it is assumed that $\mathcal{M} = \mathbb{N}$.
Thus, in particular, a $\Gamma$\emph{-definition} of a theory $\T$ is a $\Gamma$ formula 
$\tau(x)$ such that $\T = \{ \phi : \mathbb{N} \models \tau(\ulcorner \phi \urcorner) \}$.

For each $\Gamma$, $\Th_\Gamma(\mathcal{M})$ is the set of 
$\Gamma$ sentences true in $\mathcal{M}$, that is, the set 
$\{ \phi \in \Gamma : \mathcal{M} \models \phi \}$.
If $\mathcal{M}$ is a submodel of $\mathcal{N}$, and for all 
$\overline{a} \in M$ and $\gamma(\overline{x}) \in \Gamma$, $\mathcal{M} \models \gamma(\overline{a})$ 
iff $\mathcal{N} \models \gamma(\overline{a})$, then $\mathcal{N}$ is a 
$\Gamma$\emph{-elementary extension of} $\mathcal{M}$. If $\mathcal{M} \models \sigma$ 
for some $\Sigma_{n+1}$ sentence $\sigma$, and $\mathcal{N}$ is a 
$\Sigma_n$-elementary extension of $\mathcal{M}$, then $\mathcal{N} \models \sigma$.
$\mathcal{N}$ is an \emph{end-extension of} $\mathcal{M}$ (or, equivalently, 
$\mathcal{M}$ is an \emph{initial segment of} $\mathcal{N}$) iff $\mathcal{M}$ 
is a submodel of $\mathcal{N}$, and whenever $a \in M$ and $b \in N$, and 
$\mathcal{N}\models b < a$, then $b \in M$.

Let $\emptyset^{(n)}$ denote the $n$th Turing jump of the empty set 
(see, e.g.~\citet[p.~254]{rogers1967}).
For each $n$, let $\langle \varphi_i^n : i \in \omega\rangle$ be an 
acceptable (in the sense of \citet[Excercise 2.10]{rogers1967}) 
enumeration of the functions 
that are recursively enumerable (r.e.) in $\emptyset^{(n)}$; these 
are the partial $n$-recursive functions.
For each partial $n$-recursive function $\varphi_e^n$, let the $e$th 
$n$-r.e.\ set $W_e^n$ be the domain of $\varphi_e^n$. If $f$ is a function 
such that $f \simeq \varphi_e^n$ for some $e$, then $e$ is an $n$\emph{-index for} $f$.

\section{Preliminary principles}\label{section:principles}
This section presents a number of suitably formulated principles that are
useful in proving hierarchical incompleteness results of the kind given 
in Theorem~\ref{theorem:1}. 
The basic versions of these principles can be found scattered across
the literature, and none of them should 
be too surprising to the reader familiar with, e.g., \citet{hajek_pudlak1993},
\citet{kaye1991}, \citet{lindstrom2003}, and \citet{smorynski1985}.

The essence of the generalisations presented in this paper is that
r.e.\ theories are replaced by $\Sigma_{n+1}$-definable ones, and that 
the base theory is pushed down as far as it will go below $\pa$. On some 
occasions the theories have to satisfy additional constraints such as 
$\Sigma_n$-soundness or $\Pi_n$-completeness for the generalisation to go through,
and whenever this is the case that will be pointed out explicitly.

\subsection{Basics, representability, recursion theory}
Many textbooks in metamathematics rely on some version of the fact
that Robinson's arithmetic $\tq$ is $\Sigma_1$-complete: it proves
all true $\Sigma_1$ sentences. This has the consequence that
every r.e.\ set can be numerated in $\tq$ by a $\Sigma_1$ formula,
which in turn allows for representing the theorem set of $\tq$ in
$\tq$ and proving the first incompleteness theorem.
Many proofs of these facts rely on $\Delta_0$ being closed under
bounded quantification, and $\tq$ deciding all $\Delta_0$ sentences.

Since the aim of this paper is to generalise incompleteness results to
 $\Sigma_{n+1}$-definable theories rather than r.e. ones, there is a need 
 for establishing a similar correspondence between $\Sigma_{n+1}$-definable 
 sets and theories sufficient to represent them.
 On these higher levels, the 
 roles of $\Delta_0$ and $\tq$ are played by
 the classes $\Sigma_0(\Sigma_{n})$ and the theories $\ISn + \mathrm{Th}_{\Pi_n}(\mathbb{N})$, respectively. Establishing this relationship is the goal of this subsection.

The following four observations 
(see, e.g.~\citet[Lemma 0.3]{smorynski1977b}, \citet[Lemma 2.2]{hajek1977},
\citet[Lemma 2.9]{beklemishev2005}, \citet[Proposition 3.7]{kikuchi_kurahashi2017})
serves as the starting point for
the analogy between $\tq$ and $\ISn + \mathrm{Th}_{\Pi_n}(\mathbb{N})$.

\begin{fact}[Basic properties of soundness and completeness]\label{fact:soundness}
 \mbox{}
 \begin{enumerate}
  \item $\T$ is consistent iff $\T$ is $\Sigma_0$-sound.
  \item $\T$ is $\Pi_n$-complete iff $\T$ is $\Sigma_{n+1}$-complete.
  \item $\T$ is $\Sigma_n$-sound iff $\T$ is $\Pi_{n+1}$-sound iff 
        $\T + \Th_{\Sigma_{n+1}}(\mathbb{N})$ is consistent.
  \item If $\T$ is consistent and $\Pi_n$-complete, then $\T$ is $\Sigma_n$-sound.
   \end{enumerate}
\end{fact}

While the usual formulation of Gödel's theorem pertains to r.e., consistent theories, 
the hierarchically formulated Theorem \ref{theorem:1} is stated for $\Sigma_{n+1}$-definable,
$\Sigma_n$-sound theories. In light of the preceding fact, consistency and $\Sigma_0$-soundness
are equivalent, so the ordinary statement fits nicely into the hierarchical statement of the theorem.
However, inspection of the published proofs of Theorem \ref{theorem:1} 
reveals that $\Pi_n$-complete theories enter the argument in an indispensable way. In
the proof by \citet{kikuchi_kurahashi2017}, $\Pi_n$-completeness of $\T$ can be explicitly
assumed, since in the case where $\T$ is not $\Pi_n$-complete, it fails to prove all true $\Pi_n$ sentences and can hardly prove
all true $\Pi_{n+1}$ sentences. By contrast, the proof by \citet{salehi_seraji2017}
bypasses assuming $\Pi_n$-completeness of $\T$ by instead constructing a $\Pi_{n+1}$ sentence
that is independent of the $\Pi_n$-complete $\T + \mathrm{Th}_{\Pi_n}(\mathbb{N})$, a
theory whose consistency is guaranteed by the $\Sigma_n$-soundness of $\T$. That sentence
is, \emph{a fortiori}, also independent of $\T$.

The most conservative generalisation of ``r.e., consistent'' is therefore ``$\Sigma_{n+1}$-definable, $\Pi_n$-complete, and consistent'' rather than ``$\Sigma_{n+1}$-definable, $\Sigma_n$-sound''. 
As is clear
from the Theorem \ref{theorem:1}, the assumption of $\Pi_n$-completeness of $\T$ is sometimes excessively strong, and mere $\Sigma_n$-soundness is indeed enough for some further applications as well. 
In those cases, the proofs rely on the consistency of  $\T + \mathrm{Th}_{\Pi_n}(\mathbb{N})$, as in the proof by \citet{salehi_seraji2017}. 
In other cases, however, the $\Pi_n$-completeness is indispensable for the generalisation to go through.\footnote{One of the referees pointed out that even $\Pi_n$-completeness is sometimes not enough for a straightforward hierarchical generalisation to hold: see, e.g., Theorem 11 of \citet{kurahashi2018} for an example.}

The analogy with regards to soundness and completeness therefore goes as follows. 
Since $\tq$ is $\Pi_0$-complete, it is also 
$\Sigma_{1}$-complete, and since it is also consistent it follows that 
$\tq$ is $\Sigma_0$-sound and therefore also $\Pi_1$-sound. Now consider 
$\mathrm{Th}_{\Pi_n}(\mathbb{N})$: This theory is $\Pi_n$-complete and 
therefore also $\Sigma_{n+1}$-complete; it is consistent and 
therefore $\Pi_{n+1}$-sound.

The role of $\ISn$ in the analogy between
$\tq$ and $\ISn + \Th_{\Pi_n}(\mathbb{N})$ becomes clear through the next fact. Its proof
will take up most of the remainder of this subsection.

\begin{fact}\label{fact:definability}
\mbox{}

 \begin{enumerate}
 \item Every $\Sigma_n$- (or $\Pi_n$-) definable relation on $\omega$ is binumerated
        by a $\Sigma_n$ (or $\Pi_n$) formula in 
        $\tq + \Th_{\Pi_n}(\mathbb{N})$.\label{fact:definability_one}
\item Every $\Sigma_{n+1}$-definable relation on $\omega$ is 
    numerated 
        by a $\Sigma_{n+1}$ formula in $\ISn + \Th_{\Pi_n}(\mathbb{N})$.\label{fact:definability_two}
    \item Every function from $\omega^k$ to $\omega$ that is recursive in $\emptyset^{(n)}$ is strongly 
        representable by a $\Sigma_{n+1}$ formula (with particularly nice properties)
        in $\mathrm{I}\Sigma_{n} + \mathrm{exp} + \mathrm{Th}_{\Pi_n}(\mathbb{N})$.\footnote{As is well known, 
$\ISo$ proves the totality of the exponential function, so the additional axiom 
$\mathrm{exp}$ is only required in the case $n=0$ to make sure that the partial 
satisfaction predicates used to establish the particularly nice properties are well-behaved. A similar remark applies also to many of 
the following facts, and to the statements of some of the theorems in 
Section \ref{section:applications}. }
\label{fact:definability_three}
\end{enumerate}
\end{fact}
 
  The first item is immediately seen to be true, since 
  $\tq + \Th_{\Pi_n}(\mathbb{N})$ is consistent, $\Sigma_{n+1}$-complete, 
  and $\Pi_{n+1}$-sound. The second and third items are elaborated on below,   
but first we need a few more stepping stones to help out in the constructions.
  
\begin{fact}[Parametric diagonal lemma]
\label{fact:diagonal}
\mbox{}
\begin{enumerate}
\item For every $\Gamma$ formula $\gamma(x,y)$, we can effectively 
find a $\Gamma$ formula $\xi(x)$ such that
 \[
  \tq \vdash \forall x (\xi(x) \leftrightarrow \gamma(x,\ulcorner \xi\urcorner))\text{.}
 \]
 \item For every $\Gamma$ formula $\gamma(x,y)$, we can effectively 
 find a $\Gamma$ formula $\xi(x)$ such that, for each $k$,
 \[
  \tq \vdash \xi(k) \leftrightarrow \gamma(k,\ulcorner \xi(k)\urcorner))\text{.}
 \]
 \end{enumerate}

\end{fact}

\begin{bibl}
{Item 1.~of the above is essentially due to \citet[Lemma 1]{montague1962}.
The proof by \citet[Lemma~1]{ehrenfeucht_feferman1960} of 2.~actually 
suffices to prove 1.; see also \citet{smorynski1981} for a discussion of 
the development of the diagonal lemma.}
\end{bibl}

\begin{fact}[Generalised Craig's trick]\label{fact:craig}
For any $\Sigma_{n+1}$-definable theory $\T$, defined by a $\Sigma_{n+1}$ formula $\sigma(z)$, there 
  is a deductively equivalent theory $\s$, defined by a formula $\pi(x)$ 
  that is $\Pi_n$ in $\ISn$. Moreover,   
  $\ISnp \vdash \forall x (\pr_\sigma(x) \leftrightarrow \pr_\pi(x))$.\footnote{For $n=0$, 
  it is known that $\ISo$ suffices to show that the Craigified theory is deductively 
  equivalent to the original one \citep[Remark III.2.30]{hajek_pudlak1993}. 
  By contrast, this is not true of $\ID$ \citep{visser2015}.}
  \end{fact}

\begin{bibl}
{Craig's \citeyearpar{craig1953} formulation pertains to r.e.\ theories, 
for which there exist deductively equivalent theories with primitive recursive 
definitions. An early 
hierarchical generalisation is due to \citet[2.2.C]{grzegorczyk_etal_1958}.
The hierarchical formulation presented here follows by inspection of a proof by 
Kurahashi \citeyearpar[Proposition 9]{kurahashi2018}.}
\end{bibl}

\begin{conv}
{In light of Craig's trick, we may assume that any $\Sigma_{n+1}$-definable
theory is in fact $\Pi_n$-defined. In what follows, we write $\prf_\T(x,y)$ 
to denote (ambiguously) any formula $\prf_\tau(x,y)$ where $\tau$ is a 
$\Pi_n$ binumeration of $\T$ in $\ISn + \mathrm{Th}_{\Pi_n}(\mathbb{N})$, 
and moreover, this formula can be assumed to be $\Pi_n$. Consequently,
$\Pr_\T(x)$ is $\Sigma_{n+1}$ and $\con_\T$ is $\Pi_{n+1}$ in $\ISn$.}
\end{conv}

  \begin{fact}[{\citealp[Lemmata I.2.9, I.2.14, and Theorem I.2.25]{hajek_pudlak1993}}]
 \mbox{}
 \label{fact:closure}
\begin{enumerate}
 \item In $\ISn$, both $\Sigma_n$ and $\Pi_n$ are closed under bounded 
        quantifiers.\footnote{In fact, only the weaker principle of $\Sigma_n$-\emph{collection} is 
        required for this first item. However, collection plays no prominent 
        role elsewhere in this paper.}
 \item $\ISn \vdash \ISz(\Sigma_n)$.
 \item Each $\Sigma_0(\Sigma_n)$ formula is $\Delta_{n+1}$ in $\ISn$.
\end{enumerate}
\end{fact}

All the pieces are now in place to prove a, sometimes useful, lemma from which
Fact \ref{fact:definability}.\ref{fact:definability_two}
follows immediately. The proof highlights the steps where induction 
and the additional truth from the standard model is required. 
For the statement of the lemma, recall that a relation $X$ is correctly numerated
by $\phi$ in $\T$ if $\phi$ numerates $X$ in $\T$, and for all $k_1,\dots,k_n$, 
$\T \vdash \phi(k_1,\dots,k_n)$ iff $\phi(k_1,\dots,k_n)$ is true.

 \begin{lem}\label{lemma:numerability}
Let $\T$ be any $\Sigma_{n+1}$-definable, $\Pi_n$-complete, and consistent extension of 
$\ISn$ 
and let $R(x_1,\dots,x_m)$ 
be any $\Sigma_{n+1}$-definable relation. There is a $\Sigma_{n+1}$ 
formula $\phi(x_1,\dots,x_m)$ that correctly numerates $R$ in $\T$.
 \end{lem}
  
 \begin{proof}
   Let $\rho(x_1,\dots,x_m)$ define $R$. We may assume that 
   $\rho(x_1,\dots,x_m)$ is of the form $\exists z \pi(x_1,\dots, x_m,z)$  
   with $\pi \in \Pi_n$. Let $\phi(x_1,\dots,x_m)$ be such that, for all 
   $k_1,\dots,k_m$, $\ISn + \mathrm{Th}_{\Pi_n}(\mathbb{N})$ proves
  \[
    \phi(k_1,\dots,k_m) \leftrightarrow \exists z (\pi(k_1,\dots,k_m,z) \land \forall y\leq z \lnot \mathrm{Prf}_\T(\ulcorner \phi(k_1,\dots,k_m) \urcorner, y))\text{.}
  \]
  Recall that if $\T$ is $\Sigma_{n+1}$-definable, then there is a deductively 
  equivalent $\Pi_n$ definition of $\T$, and $\mathrm{Prf}_\T(x,y)$ is 
  therefore equivalent to a $\Pi_n$ formula in $\ISn$. Then
  \[
  \pi(k_1,\dots,k_m,z) \land \forall y\leq z \lnot \mathrm{Prf}_\T(\ulcorner \phi(k_1,\dots,k_m) \urcorner, y)
  \]
  is a $\Sigma_0(\Sigma_n)$ formula, and is therefore, by Fact \ref{fact:closure},
  equivalent to a $\Delta_{n+1}$ formula in $\ISn$. It follows that $\phi$ is 
  equivalent to a $\Sigma_{n+1}$ formula in $\ISn$.

  Suppose $R(k_1,\dots,k_m)$. Then $\mathbb{N} \models \rho(k_1,\dots,k_m)$, 
  so there is then an $i$ such that $\mathbb{N} \models \pi(k_1,\dots,k_m,i)$. 
 Since $\pi$ is $\Pi_n$ and $\mathrm{Th}_{\Pi_n}(\mathbb{N})$ is $\Pi_n$-complete, it follows that
  $\ISn + \mathrm{Th}_{\Pi_n}(\mathbb{N}) \vdash \pi(k_1,\dots,k_m,i)$ .
  Suppose, for a contradiction, that 
  $\T \nvdash \phi(k_1,\dots,k_m)$. By Fact \ref{fact:definability}.\ref{fact:definability_one}, 
  \[
  \ISn + \mathrm{Th}_{\Pi_n}(\mathbb{N}) \vdash \lnot \mathrm{Prf}_\T(\ulcorner \phi(k_1,\dots,k_m)\urcorner,p)
  \]
  for all $p$.
  It follows that 
  \[
  \ISn + \mathrm{Th}_{\Pi_n}(\mathbb{N}) \vdash \forall y \leq i \lnot \mathrm{Prf}_\T(\ulcorner \phi(k_1,\dots,k_m) \urcorner, y)\text{,}
  \]
  so $\ISn + \mathrm{Th}_{\Pi_n}(\mathbb{N}) \vdash \phi(k_1,\dots,k_m)$, 
  and $\T \vdash \phi(k_1,\dots,k_m)$.

  Conversely, suppose $\T \vdash \phi(k_1,\dots,k_m)$, and let $p$ be a proof 
  of $\phi(k_1,\dots,k_m)$ in $\T$. Then 
  $\ISn + \mathrm{Th}_{\Pi_n}(\mathbb{N}) \vdash \mathrm{Prf}_\T(\ulcorner \phi(k_1,\dots,k_m)\urcorner, p)$.
  It follows that 
  \[
  \ISn + \mathrm{Th}_{\Pi_n}(\mathbb{N}) \vdash \forall y \leq z\lnot \mathrm{Prf}_\T(\ulcorner \phi(k_1,\dots,k_m)\urcorner, y) \rightarrow z < p\text{.}
  \]
  Suppose, for a contradiction, that $\lnot R(k_1,\dots,k_m)$. Then for all $i$, 
  \[
  \ISn + \mathrm{Th}_{\Pi_n}(\mathbb{N}) \vdash \lnot \pi(k_1,\dots,k_m,i)\text{.}
  \]
  It follows that $\ISn + \mathrm{Th}_{\Pi_n}(\mathbb{N}) \vdash \lnot \exists z < p \pi(k_1,\dots,k_m,z)$. 
  We get 
  \[
  \ISn + \mathrm{Th}_{\Pi_n}(\mathbb{N}) \vdash \lnot \exists z (\pi(k_1,\dots,k_m, z) \land \forall y \leq z \lnot \mathrm{Prf}_\T(\ulcorner \phi(k_1,\dots,k_m)\urcorner, y))\text{,}
  \]
  whence $\ISn + \mathrm{Th}_{\Pi_n}(\mathbb{N}) \vdash \lnot \phi(k_1,\dots,k_m)$, 
  and $\T \vdash \lnot \phi(k_1,\dots,k_m)$.
  
  It remains to show that $\T \vdash \phi(k_1,\dots,k_m)$ iff $\phi(k_1,\dots,k_m)$ 
  is true. The implication from right to left is trivial, since $\phi$ is 
  $\Sigma_{n+1}$ in $\ISn$ and $\mathrm{Th}_{\Pi_n}(\mathbb{N})$ is 
  $\Sigma_{n+1}$-complete.  For the other direction, suppose 
  $\T \vdash \phi(k_1,\dots,k_m)$, and let $p$ be the least such proof. Then 
  \[
  \ISn + \mathrm{Th}_{\Pi_n}(\mathbb{N}) \vdash \forall y \leq z\lnot \mathrm{Prf}_\T(\ulcorner \phi(k_1,\dots,k_m)\urcorner, y) \rightarrow z < p\text{.}
  \]

  Suppose further that there is no $i < p$ such that $\pi(k_1,\dots,k_m,i)$. 
  Then, as before, $\T \vdash \lnot \phi(k_1,\dots,k_m)$, a contradiction. 
  Thus there is an $i < p$ such that $\pi(k_1,\dots,k_m,i)$ is true. Since 
  $p$ is minimal, 
  $\forall y \leq i \lnot \mathrm{Prf}_\T(\ulcorner \phi(k_1,\dots,k_m)\urcorner,y)$, so 
  \[
  \exists z (\pi(k_1,\dots,k_m,z) \land \forall y\leq z \lnot \mathrm{Prf}_\T(\ulcorner \phi(k_1,\dots,k_m) \urcorner, y))\]
  is true, and therefore $\phi(k_1,\dots,k_m)$ is true, as desired.
 \end{proof}

 \begin{bibl}
 {The correct representability of $\Sigma_1$ relations in $\tq$ is due to 
 \cite{ehrenfeucht_feferman1960}. The proof above mimics Lindstr\"om's 
 \citeyearpar{lindstrom2003} version of Shepherdson's \citeyearpar{shepherdson1961}
 proof of the same result.}
   \end{bibl}

Fact \ref{fact:definability}.\ref{fact:definability_two} follows directly from Lemma 
 \ref{lemma:numerability}, and    
it only remains to give an argument for Fact \ref{fact:definability}.\ref{fact:definability_three}. 
For this we need two more facts, the first being a version of Post's theorem.

\begin{fact}[\citealp{post1948}]
\label{fact:post}
  A relation is $\Sigma_{n+1}$ iff it is r.e.\ in $\emptyset^{(n)}$.
\end{fact}

\begin{fact}[The selection theorem]
 \label{fact:selection}
  For each $\Sigma_{n+1}$ formula $\phi$ with exactly the variables 
  $x_1,\dots,x_k$ free, there is a   $\Sigma_{n+1}$ formula $\mathrm{Sel}\{\phi\}$ 
  with exactly the same free variables, such that:
  \begin{enumerate}
\item $\mathrm{I}\Sigma_n \vdash \forall
  x_1,\dots,x_k(\mathrm{Sel}\{\phi\}(x_1,\dots,x_k) \rightarrow
  \phi(x_1,\dots,x_k))$;
 \item $\mathrm{I}\Sigma_n \vdash \forall x_1,\dots,x_k,z(\mathrm{Sel}\{\phi\}(x_1,\dots,x_k) \land\mathrm{Sel}\{\phi\}(x_1,\dots,x_{k-1},z) \rightarrow x_k = z)$;
 
\item $\mathrm{I}\Sigma_n \vdash \forall x_1,\dots,x_{k-1} (\exists x_k \phi(x_1,\dots,x_k) \rightarrow \exists x_k \mathrm{Sel}\{\phi\}(x_1,\dots,x_k))$.
\end{enumerate}
\end{fact}
\begin{proof}[Proof sketch]
We may assume that any given $\Sigma_{n+1}$ formula $\phi(x_1,\dots,x_k)$ is on the form 
$\exists w \pi(x_1,\dots,x_k,w)$ with $\pi \in \Pi_n$. Let $(w)_i$ denote the $i$th element of
the ordered pair coded by $w$, and define $\mathrm{Sel}\{\phi\}(x_1,\dots,x_k)$ 
to be the formula
\[
 \exists w (\pi(x_1,\dots,x_{k-1},(w)_0,(w)_1) \land \forall u \leq w \lnot \pi(x_1,\dots,x_{k-1},(u)_0,(u)_1) \land (w)_0 = x_k)\text{.}
\]
By Fact \ref{fact:closure}, this formula is equivalent in $\ISn$ to a 
$\Sigma_{n+1}$ formula. Using $\Sigma_n$-induction, it is easy to
check that it has the three desired properties.
\end{proof}

 \begin{bibl}
{\citet[Theorem 0.6.9]{smorynski1985} provides a proof of the selection 
theorem for $n=0$, and the generalisation is straightforward. The numeration 
below of partial recursive functions is also based on the treatment in Section~0 of 
his book.}
 \end{bibl}

Fact \ref{fact:definability}.\ref{fact:definability_three} can now be established as follows:
Let $f$ be a $k$-ary partial $n$-recursive function; then the relation 
$f(x_1,\dots,x_k) = y$ is $\Sigma_{n+1}$ by Fact \ref{fact:post}.
By Fact \ref{fact:definability}.\ref{fact:definability_two}, this relation is 
correctly numerated in $\ISn + \mathrm{Th}_{\Pi_n}(\mathbb{N})$ by some 
$\Sigma_{n+1}$ formula $\psi(x_1,\dots,x_k,y)$.
Finally, by Fact \ref{fact:selection}, $f$ is strongly represented in 
$\ISn + \mathrm{exp} + \mathrm{Th}_{\Pi_n}(\mathbb{N})$ by the $\Sigma_{n+1}$ formula 
$\mathrm{Sel}\{\Sat_{\Sigma_{n+1}}\}(\ulcorner \psi \urcorner, x_1,\dots,x_k, y)$.
This concludes the proof of Fact \ref{fact:definability}.

\begin{conv}
{The partial $n$-recursive function with index $e$, $\varphi_e^n$, can now be 
defined to be the function whose graph is defined by 
$\mathrm{Sel}\{\Sat_{\Sigma_{n+1}}\}(e, y_1, \dots, y_k, z)$ in $\mathbb{N}$. 
The resulting enumeration is acceptable in the sense of \citet{rogers1967}. 
Whenever convenient, $\varphi_e^n(m_1,\dots,m_i) = k$ is used as a shorthand 
for $\mathrm{Sel}\{\Sat_{\Sigma_{n+1}}\}(e, m_1,\dots, m_i, k)$.}
\end{conv}

To wrap up this subsection, we sketch a proof of a formalised version of the second recursion theorem.

\begin{fact}[Formalised second recursion theorem]\label{fact:recursion}
 Let $f : \omega^2 \to \omega$ be $n$-recursive. There is an $e \in \omega$ such that
 \[
  \ISn + \mathrm{exp} + \mathrm{Th}_{\Pi_n}(\mathbb{N}) \vdash \varphi^n_e(y) \simeq f(e,y)\text{.}
 \]
\end{fact}

\begin{proof}
 Let, by Fact \ref{fact:definability}.3, $\psi(x,y,z)$ be a $\Sigma_{n+1}$ formula 
 strongly representing $f$ in $\ISn + \mathrm{exp} + \mathrm{Th}_{\Pi_n}(\mathbb{N})$. 
Let, by Fact \ref{fact:diagonal}, $\gamma(y,z)$ be a formula such that
 \[
  \ID \vdash \gamma(y,z) \leftrightarrow \mathrm{Sel} \{ \mathrm{\Sat_{\Sigma_{n+1}}} \} (\ulcorner \psi \urcorner, \ulcorner \gamma \urcorner, y,z)
 \]
Then $e= \ulcorner \gamma \urcorner$ is as desired.
\end{proof}

The recursion theorem is usually deployed in the following manner. Define an 
$n$-recursive function $f(z,x)$ in stages, using $z$ as a parameter; the 
resulting function may differ depending on the choice of $z$. By the recursion theorem, 
there is then an $n$-index $e$ such that $\varphi_e^n(x)$ coincides with $f(e,x)$. 
This legitimates self-referential constructions where an index of $f$ is being 
used in the construction of $f$ itself.

 \begin{bibl}
{The recursion theorem is due to \citet[Theorem XXVII]{kleene1952}. 
\citet[Theorem 0.6.12]{smorynski1985} shows how it can be formalised in 
(essentially) $\ID$ for ($0$-)recursive functions.}
 \end{bibl}
   
 \subsection{Strong provability predicates}
 A central piece in the proof of Gödel's incompleteness theorem for r.e.\
 theories $\T$ is the use of formal provability predicates $\pr_\T(x)$, expressing 
 ``$x$ is provable in $\T$''.
 In the current setting, with $\Sigma_{n+1}$-definable theories in focus, the 
 corresponding \emph{strong provability predicates} $\Pr_{\T,\Sigma_{n+1}}(x)$ express ``$x$ is provable in $\T$ from true $\Sigma_{n+1}$ sentences''.\footnote{Similar 
provability predicates, $\Pr_{\T,\Pi_n}(x)$ appear in the literature under the 
name strong provability, $n$-provability, and oracle provability 
\citep{ignatiev1993,beklemishev2005,visser2015,kolmakov_beklemishev2019}.
The difference is usually only a 
matter of taste, since the theories $\T + \mathrm{Th}_{\Pi_n}(\mathbb{N})$ and 
$\T + \mathrm{Th}_{\Sigma_{n+1}}(\mathbb{N})$ are deductively equivalent, and 
under reasonable assumptions this is reflected also in the formalised notions 
$\pr_{\T,\Pi_{n}}(x)$ and $\pr_{\T,\Sigma_{n+1}}(x)$. However, the relationship 
between the related notions $\pr^k_{\T,\Pi_{n}}(x)$ and $\pr^k_{\T,\Sigma_{n+1}}(x)$ 
(introduced below) is not as immediate, since there $k$ bounds the length of the proof 
of $x$ and the Gödel number of the additional true $\Pi_n$ or $\Sigma_{n+1}$ 
sentence used in the proof. Even though every $\Sigma_{n+1}$-provable sentence 
has a $\Pi_n$-proof, this proof might be much longer than the original one. We 
opt for the $\Sigma_{n+1}$ versions, since it helps in the proof Fact \ref{fact:OHGL} 
below, and since it makes (some of) the indices align nicely.}
  To define these predicates, we first need to introduce partial satisfaction predicates.

\begin{fact}[Partial satisfaction predicates]\label{fact:satisfaction}
 For each $k$ and $\Gamma$, there is a $k+1$-ary $\Gamma$ formula 
 $\Sat_{\Gamma}(x,x_1,\dots,x_k)$ such that for every $\Gamma$ formula 
 $\phi(x_1,\dots,x_k)$,
 \[
  \ID \vdash \forall x_1,\dots,x_k(\phi(x_1,\dots,x_k) \leftrightarrow \Sat_{\Gamma}(\ulcorner \phi \urcorner, x_1,\dots,x_k))\text{.}
 \]
  Hence there is also a $\Gamma$ formula $\tr_\Gamma(x)$ such that for 
  every $\Gamma$ sentence $\phi$,
 \[
  \ID \vdash \phi \leftrightarrow \tr_\Gamma(\ulcorner \phi \urcorner)\text{.}
 \]
\end{fact}

\begin{bibl}
{Modern proofs of this Fact are due to \citet{kaye1991} and \citet{hajek_pudlak1993}. The use of partial satisfaction predicates, however, goes back to \citet{hilbert_bernays1934}. 
}
\end{bibl}

\begin{definition}
Let, for each $\Gamma$, $\prf_{\T,\Gamma}(x,y)$ be the formula 
\[
\exists z(z  \in \Gamma  \land \tr_{\Gamma}(z) \land \prf_{\T+z}(x,y))\text{.}
\] 
Let $\pr_{\T,\Gamma}(x) \defeq \exists y \prf_{\T,\Gamma}(x,y)$, and let 
$\con_{\T,\Gamma} \defeq \lnot \pr_{\T,\Gamma}(\ulcorner \bot \urcorner)$.
\end{definition}
We are mainly interested in the predicates $\pr_{\T,\Sigma_{n+1}}(x)$.
These provability predicates 
have many similarities with the usual provability predicate. Firstly, if $\T$ 
is $\Sigma_{n+1}$-definable, then $\pr_{\T,\Sigma_{n+1}}(x)$ is $\Sigma_{n+1}$, 
and $\con_{\T,\Sigma_{n+1}}$ consequently $\Pi_{n+1}$. Secondly, they satisfy 
provable $\Sigma_{n+1}$-completeness and Löb's derivability conditions. These 
notions also have their roots with \citet{hilbert_bernays1934}; see also \citet[Theorem 5.4 and Corollary 5.5]{feferman1960}. The 
versions presented here are extracted from \citet[Propositions 2.10 and 2.11]{beklemishev2005} 
and \citet[Lemma 3.3.7]{smorynski1985}.

\begin{fact}[Provable $\Sigma_{m+1}$-completeness]\label{fact:completeness}
Let $\T$ be a $\Sigma_{n+1}$-definable theory, and let $\sigma(x_1,\dots,x_k)$ 
be any $\Sigma_{m+1}$ formula. Then
\[
\ID \vdash \forall x_1,\dots,x_k(\sigma(x_1,\dots,x_k)\rightarrow \pr_{\T,\Sigma_{m+1}}(\ulcorner \sigma(\dot{x}_1,\dots,\dot{x}_k)\urcorner))\text{.}
\]
\end{fact}

\begin{fact}[L\"ob conditions]\label{fact:lob}
If $\T$ is a $\Sigma_{n+1}$-definable extension of $\ID$,
then, for each $m\geq 0$, and for all sentences $\phi$, $\psi$,
\begin{enumerate}
\renewcommand{\theenumi}{L\arabic{enumi}}
\item if $\T\vdash \phi$, then $\ISn + \exp + \mathrm{Th}_{\Pi_n}(\mathbb{N}) \vdash \pr_{\T,\Sigma_{m+1}}(\ulcorner \phi \urcorner)$;
 \item $\ID  \vdash \pr_{\T,\Sigma_{m+1}}(\ulcorner \phi \urcorner) \land \pr_{\T,\Sigma_{m+1}}(\ulcorner \phi \rightarrow \psi\urcorner) \rightarrow \pr_{\T,\Sigma_{m+1}}(\ulcorner \psi \urcorner)$;
 \item $\ID \vdash \pr_{\T,\Sigma_{m+1}}(\ulcorner \phi \urcorner) \rightarrow \pr_{\T,\Sigma_{m+1}}(\ulcorner \pr_{\T,\Sigma_{m+1}}(\ulcorner \phi \urcorner)\urcorner)$.
\end{enumerate}

Similar statements also hold for formulae:
\begin{enumerate}
\renewcommand{\theenumi}{L\arabic{enumi}'}
\item if $\T\vdash \forall x \phi(x)$, then $\ISn + \exp + \mathrm{Th}_{\Pi_n}(\mathbb{N}) \vdash \forall x \pr_{\T,\Sigma_{m+1}}(\ulcorner \phi(\dot{x}) \urcorner)$;
\item $\ID  \vdash \forall x (\pr_{\T,\Sigma_{m+1}}(\ulcorner \phi(\dot{x}) \urcorner) \land \pr_{\T,\Sigma_{m+1}}(\ulcorner \phi(\dot{x}) \rightarrow \psi(\dot{x})\urcorner)\\\phantom{\ID  \vdash \,}\rightarrow \pr_{\T,\Sigma_{m+1}}(\ulcorner \psi(\dot{x}) \urcorner))$;
 \item $\ID \vdash \forall x (\pr_{\T,\Sigma_{m+1}}(\ulcorner \phi(\dot{x}) \urcorner) \rightarrow \pr_{\T,\Sigma_{m+1}}(\ulcorner \pr_{\T,\Sigma_{m+1}}(\ulcorner \phi(\dot{x}) \urcorner)\urcorner))$.
\end{enumerate}

\end{fact}
The stronger background theory used in items L1.\ and L1'.\ is enough to ensure the
numerability of the $\Sigma_{n+1}$-definable theory $\T$. For r.e.\ theories $\T$, 
$\ID$ suffices. 

We now turn our attention to the \emph{bounded provability predicates} that 
feature prominently in the sequel. Consider again the $\Sigma_{n+1}$ 
formula $\pr_{\T,\Sigma_{n+1}}(x)$ for a $\Sigma_{n+1}$-definable $\T$. 
We may assume that there is a $\Pi_n$ formula $\pi(x,y)$ such that 
\[
 \ISn \vdash \pr_{\T,\Sigma_{n+1}}(x) \leftrightarrow \exists y \pi(x,y)\text{.}
\]
Let $\pr_{\T,\Sigma_{n+1}}^k(x)$ be the formula $\exists z \leq k \pi(x,z)$.\footnote{In 
the notation of, e.g., \citet{lindstrom_shavrukov2008}, $\pr^k_{\T,\Sigma_{n+1}}(x)$ 
would be written $k : \pr_{\T,\Sigma_{n+1}}(x)$.} 
Then, by Fact \ref{fact:closure}, this formula is $\Pi_n$ in $\ISn$, 
and therefore decidable in 
$\ISn + \mathrm{exp} + \mathrm{Th}_{\Pi_n}(\mathbb{N})$.\footnote{The additional axiom 
$\mathrm{exp}$ is again only required for $n=0$, to handle
the partial truth definition occuring in $\pr_{\T,\Sigma_{n+1}}$. 
Remarks of this type will hereafter be omitted.} As a consequence, we have strong 
reflection properties for the bounded proof predicates.

\begin{fact}[Uniform small reflection]\label{fact:exp_reflection}
Let $\T$ be a $\Sigma_{n+1}$-definable, consistent extension of $\ID$. For each 
$\phi(x)$ and $k$, we have:
 \[
  \ISn + \mathrm{exp} + \mathrm{Th}_{\Pi_{n}(\mathbb{N})}\vdash \forall x (\pr_{\T,\Sigma_{n+1}}^k(\ulcorner\phi(\dot{x})\urcorner) \rightarrow \phi(x))\text{.}
 \]
\end{fact}

\begin{bibl}
{A forerunner to this reflection principle is proved by \citet[Lemma 2.18]{feferman1962}. 
The generalisation to $\Sigma_{n+1}$-definable theories is straightforward; see also 
\citet[Lemma III.4.40]{hajek_pudlak1993} for some middle ground: reflection for proofs 
from true $\Sigma_{n+1}$ sentences for r.e.\ theories.}
\end{bibl}

\begin{fact}[Formalised small reflection]\label{fact:reflection}
Let $\T$ be a $\Sigma_{n+1}$-definable, consistent extension of $\ISnp$. Then we have:
 \[
  \ID \vdash \forall \phi \forall z\pr_{\T,\Sigma_{n+1}}(\ulcorner \pr_{\T,\Sigma_{n+1}}^{\dot{z}}(\ulcorner \phi \urcorner) \rightarrow \phi\urcorner)\text{.}
 \]
\end{fact}

\begin{bibl}
{\citet{verbrugge_visser1994} show how the small reflection principle can be 
formalised in (theories weaker than) $\ID$. I am grateful to one of the refeees 
for pointing out a crucial error in an earlier statement of this Fact.}
\end{bibl}

\subsection{Model theory of arithmetic}

The remainder of this section concerns the model theory of arithmetic, 
building up to a characterisation of $\Pi_n$-conservativity in the spirit 
of Orey, H\'{a}jek, Guaspari, and Lindstr\"om. A first step towards that 
goal is the following miniaturisation of the arithmetised completeness 
theorem in the style of \citet[Theorems 1.7 and 2.2]{mcaloon1978}. 

 \begin{fact}[The arithmetised completeness theorem]\label{fact:ACT}
  Fix $m\leq n$. If $\mathcal{M} \models \ISnp$, and $\T$ is a theory that is
  $\Sigma_{n+1}$-definable in $\mathcal{M}$ such that $\mathcal{M} \models \con_{\T,\Pi_m}$, 
  then there is a $\Sigma_m$-elementary end-extension of $\mathcal{M}$ satisfying $\T$.
 \end{fact}

The proof of the arithmetised completeness theorem rests on the following version of the low basis theorem by \citet[Corollary I.3.10(1)]{hajek_pudlak1993}: 
 \begin{fact}[Low basis theorem]\label{fact:lowbasis}
  Provably in $\ISnp$, each dyadic unbounded $\Delta_{n+1}$ tree has an unbounded $LL_{n+1}$ branch.
 \end{fact}
Here, a \emph{tree} is a set of finite binary sequences that is closed under taking initial segments. A \emph{branch} through a tree is a subtree that is linearly ordered under the relation ``being an initial segment of''. See \citet[Chapter I.3(b)]{hajek_pudlak1993} for more details.
  A definition of the class $LL_{n+1}$ used in the statement of the low basis theorem can be found in \citet[Chapter I.2(d)]{hajek_pudlak1993}. The only properties of $LL_{n+1}$ sets that are used in the proof of the arithmetised completeness theorem is that $\ISnp$ proves induction for $\Sigma_1(LL_{n+1})$ sets, and that every set recursive in an $LL_{n+1}$ set is itself $LL_{n+1}$ \citep[I.2.78--79]{hajek_pudlak1993}. 

For the proof of the arithmetised completeness theorem, we also rely on $\ISnp$ being able to define formalised versions of syntactic and semantic notions such as ``formula'', ``term'', ``theory'', ``satisfaction'', and ``model'', so that the relevant constructions can be carried out within $\ISnp$ itself. The reader is again referred to \citet{hajek_pudlak1993}, especially Chapter I.4, for a detailed development of these concepts. The generalisation to $\Sigma_{n+1}$-definable theories is straightforward, and fits safely within $\ISnp$.

\begin{proof}[Proof of the arithmetised completeness theorem]
  Fix $m\leq n$. Let $\mathcal{M} \models \ISnp$ and let $\T$ be a theory $\Sigma_{n+1}$-definable in $\mathcal{M}$ such that $\mathcal{M}\models \con_{\T,\Pi_{m}}$. Reason in $\mathcal{M}$:
  
  Since $\T$ is $\Sigma_{n+1}$ and we have $\ISnp$, we may assume $\T$ to be $\Pi_n$-defined and Henkinised. Let $\psi_0,\psi_1,\dots$ be an enumeration of all sentences. Define a dyadic tree $T$ by
  \begin{quote}
  $s \in T$ iff there is no $p \leq s $ such that $p$ is a proof of contradiction in $\T$ from the true $\Pi_{m}$ sentences plus $\{ \psi_i^{(s)_i} : i < l(s)\}$.
  \end{quote}
  The proof relation for $\T + \mathrm{Th}_{\Pi_{m}}(\mathcal{M})$ is $\Pi_{n}$ in $\mathcal{M}$, so the tree $T$ is at most $\Delta_{n+1}$ in $\mathcal{M}$. Then $\ISnp$ suffices to show that $T$ is unbounded, so by Fact \ref{fact:lowbasis} there is an unbounded $LL_{n+1}$ branch $B = \{b_0,b_1,\dots\}$ through $T$. 
  
  Let $\hat{\T} = \{ \psi_i : b_i = 1\}$.
  Then $\hat{\T}$ is 
 recursive in the $LL_{n+1}$ branch $B$,
  and is therefore $LL_{n+1}$ itself.
  A term model $\mathcal{K}$ satisfying $\T + \mathrm{Th}_{\Pi_{m}}(\mathcal{M})$ can then be read off $\hat{\T}$ in the usual way.
Finally, note that $\mathcal{K}$ is recursive in $\hat{\T}$ and therefore $LL_{n+1}$.
Using induction for $\Sigma_1(LL_{n+1})$ (provided by $\ISnp$), we can now define an $LL_{n+1}$ embedding $f$ of $\mathcal{M}$ onto an initial segment of $\mathcal{K}$ by letting $f(0) = 0^\mathcal{K}$ and $f(x+1) =$ the $\mathcal{K}$-successor of $f(x)$.
 \end{proof}

The next few facts are used in the proof of the generalised Orey-Hájek characterisation.

\begin{fact}[Overspill]
 Let $\mathcal{M} \models \ISnp$. Suppose that $a \in M$ and that $\phi(x,y)$ 
 is a $\Sigma_0(\Sigma_{n+1})$ formula such that $\mathcal{M} \models \phi(k,a)$ 
 for all $k\in\omega$. Then there is a $b \in M \setminus \omega$ such that 
 $\mathcal{M} \models \forall x \leq b \phi(x,a)$.
\end{fact}

\begin{bibl}
 {The notion of overspill is originally due to \cite{robinson1963}, while the 
 hierarchical version stated here is from \citet[Corollary IV.1.16]{hajek_pudlak1993}.}
\end{bibl}

\begin{fact}\label{fact:code}
 If $\mathcal{M}$ is a non-standard model of $\ISnp$, and $\phi(x)$ is a $\Sigma_{n+1}$ 
 formula which may include parameters from $\mathcal{M}$, then $\{ k \in \omega : \mathcal{M} \models \phi(k)\}$ is coded in $\mathcal{M}$.
\end{fact}

It follows that if $\mathcal{M} \models \ISnp$, then $\Th_{\Sigma_{n+1}}(\mathcal{M})$ 
is coded in $\mathcal{M}$.

\begin{fact}[\citealt{mcaloon1982}, cf.\ \citealt{daquino1993}]\label{fact:mcaloon}
 If $\mathcal{M}$ is a countable non-standard model of $\mathrm{I}\Delta_0$, 
 and $\T$ is $\Sigma_1$-sound, then for every non-standard  $c \in M$, there 
 is a non-standard initial segment of $\mathcal{M}$ below $c$ that is a model of $\T$.
\end{fact}

\begin{fact}[Refined Friedman embedding theorem]\label{fact:friedman}
 If $\mathcal{M}$ and $\mathcal{N}$ are countable non-standard models of $\ISnp$
then the following are equivalent:
 \begin{enumerate}
  \item $\mathcal{M}$ is embeddable as a $\Sigma_n$-elementary initial segment of $\mathcal{N}$;
  \item $\SSy(\mathcal{M}) = \SSy(\mathcal{N})$ and 
  $\Th_{\Sigma_{n+1}}(\mathcal{M}) \subseteq \Th_{\Sigma_{n+1}}(\mathcal{N})$.
 \end{enumerate}
\end{fact}

\begin{bibl}
{This refinement of Friedman's \citeyearpar{friedman1973} embedding theorem for 
$n=0$ is due to \citet[Theorem 1.I]{ressayre1987} and \citet[Corollary 2.4]{dimitracopoulos_paris1988}, independently. The hierarchical generalisation is straightforward, and has 
been worked out by \citet[Corollary 15]{cornaros}.}
\end{bibl}
We are now ready to prove the final fact:
an excerpt of a generalisation of the Orey-H\'{a}jek-Guaspari-Lindstr\"om 
characterisation of interpretability. For extensions of $\pa$, the equivalence 
of 1.\ and 3.\ is due to Guaspari \citeyearpar[Theorem 6.5(1)]{guaspari1979}. 
The equivalence of 1.\ and 2.\ for finitely axiomatisable theories seems to have 
been known to experts for some time, while the equivalence of 2.\ and 3.\ for 
r.e.\ extensions of fragments of $\pa$ is stated without proof by \citet[Theorem 2.11]{blanck_enayat2017}. 
With the previous facts of this section in place,
the generalisation to $\Sigma_{n+1}$-definable theories presents no further difficulties.
 
 \begin{fact}[OHGL characterisation]\label{fact:OHGL}
Let $\s$ and $\T$ be $\Sigma_{n+1}$-definable, consistent extensions of $\ISnp$, and 
suppose that $\s$ is also $\Pi_n$-complete. The following are equivalent:
\begin{enumerate}
  \item $\s$ is $\Pi_{n+1}$-conservative over $\T$;
 \item for all $k\in\omega$, $\T \vdash \con_{\s,\Sigma_{n+1}}^k$;
 \item every countable model $\mathcal{M}$ of $\T$ with $\s \in \SSy(\mathcal{M})$ 
        has a $\Sigma_n$-elementary extension to a model of $\s$.
\end{enumerate}
\end{fact}

\begin{proof}
 1. $\Rightarrow$ 2. Suppose that $\s$ is $\Pi_{n+1}$-conservative over $\T$. By 
 Fact \ref{fact:exp_reflection} and $\Pi_n$-completeness of $\s$, 
 $\s \vdash \con_{\s,\Sigma_{n+1}}^k$ for all $k\in \omega$. But 
 $\con_{\s,\Sigma_{n+1}}^k$ is at most $\Pi_{n+1}$, so $\T \vdash \con_{\s,\Sigma_{n+1}}^k$ 
 for all $k \in \omega$.
 
 2. $\Rightarrow$ 3. Suppose $\T \vdash \con_{\s,\Sigma_{n+1}}^k$ for all 
 $k \in \omega$. Let $\mathcal{M}$ be a countable model of $\T$ and suppose 
 that $\s \in \SSy(\mathcal{M})$.
 Since $\T$ extends $\ISnp$, 
 it follows that $\con_{\s,\Sigma_{n+1}}^k$ is at most $\Pi_{n+1}$, 
 and we can use overspill to get $\mathcal{M} \models \con_{\s,\Sigma_{n+1}}^c$ 
 for some non-standard $c\in M$. By Fact \ref{fact:mcaloon}, there is a 
 submodel $\mathcal{M}_0 \models \pa$ of $\mathcal{M}$, all of whose elements 
 are below $c$, and there is some non-standard $a$ below $c$ that codes 
 $\Th_{\Sigma_{n+1}}(\mathcal{M})$. This ensures that 
 $\mathcal{M}_0 \models \con_{\s + \{m:m\varepsilon a\}}$, so Fact \ref{fact:ACT} 
 guarantees the existence of an end-extension $\mathcal{K}$ of $\mathcal{M}_0$, 
 satisfying $\s + \{m:m\varepsilon a\}$ and therefore also $\s + \Th_{\Sigma_{n+1}}(\mathcal{M})$. 
 
 At this point, the situation is that $\SSy(\mathcal{M}) = \SSy(\mathcal{M}_0) = \SSy(\mathcal{K})$, 
 $\mathcal{M}$ and $\mathcal{K}$ are countable, and 
 $\Th_{\Sigma_{n+1}}(\mathcal{M}) \subseteq \Th_{\Sigma_{n+1}}(\mathcal{K})$.
  Then 
 Fact \ref{fact:friedman} ensures that $\mathcal{M}$ can be embedded as a 
 $\Sigma_n$-elementary initial segment of $\mathcal{K}$.

 3. $\Rightarrow$ 1. Prove the contrapositive statement by assuming that $\s$ is 
 not $\Pi_{n+1}$-conservative over $\T$. Then there is a $\Pi_{n+1}$ sentence 
 $\pi$ such that $\s \vdash \pi$ but $\T+ \lnot \pi$ is consistent. Let 
 $\mathcal{M} \models \T+ \lnot \pi$
 . If $\mathcal{K} \models \s$ were a $\Sigma_n$-elementary extension of 
 $\mathcal{M}$, then $\mathcal{K}$ would satisfy both $\pi$ and $\lnot \pi$, 
 a contradiction.
\end{proof}

\section{Applications}\label{section:applications}
The goal of this section is to prove a handful of hierarchical incompleteness 
results, using the tools we reviewed in the previous one.
The first such result stems from \citet[Theorem 2]{mostowski1961}, who proved 
that whenever $\{\T_i: i \in \omega\}$ is an r.e.\ family of consistent, 
r.e.\ theories extending $\tq$, then there is a $\Pi_1$ formula that is 
simultaneously independent over these theories. Here, we understand the 
concept of an independent formula in the following way:

\begin{definition}
 A formula $\xi(x)$ is \emph{independent} over $\T$ if, for every 
 $g : \omega \to \{0,1\}$, the theory $\T + \{ \xi(k)^{g(k)}\}$ is consistent. 
 Recall that $\xi(k)^0= \lnot \xi(k)$ and $\xi(k)^1 = \xi(k)$.
\end{definition}

While one of Mostowski's accomplishments was the simultaneous independence 
over a whole r.e.\ family of theories, this aspect of his result is 
deliberately ignored here. Instead, we focus on how to construct formulae 
independent over $\Sigma_{n+1}$-definable theories. 

\begin{thm}\label{theorem:independence}
 Let $\T$ be a $\Sigma_{n+1}$-definable, $\Sigma_n$-sound 
 extension of $\ISn + \mathrm{exp}$. Then there is a $\Sigma_{n+1}$ 
 formula $\xi(x)$ that is independent over $\T$.
\end{thm}

\begin{proof}
 Define a function $f(x)$ by the stipulation that 
 $f(m) = k$ iff there is a proof $p$ of $\varphi^n_m(m) \neq k$ in $\T$, 
 and for each $q<p$ and $k_0 \leq k$, $q$ is not a proof of $\varphi^n_m(m) \neq k_0$ in $\T$.
 
 Since $\T$ is $\Sigma_{n+1}$-definable, there is a deductively equivalent 
 $\Pi_n$-definition of $\T$. Hence, the relation $f(x) = y$ is r.e.\ in 
 $\emptyset^{(n)}$, and therefore partial $n$-recursive.
 Let $e$ be an $n$-index for $f$, and let $\xi(x)$ be the formula 
\[
\exists z (\varphi^n_e(e) = z \land \mathrm{Sat}_{\Sigma_{n+1}}(z,x))\text{.}
\]
 Since $\T$ extends $\ISn$, we may assume that both $\varphi^n_x(y) = z$ and 
 $\xi(x)$ are equivalent in $\T$ to $\Sigma_{n+1}$ formulae. 
 The proof that $\xi(x)$ is as desired has two parts. The first part shows that
 $\T + \varphi^n_e(e) = k$ is consistent for each $k\in\omega$.

Suppose, for a contradiction, that $\T + \varphi^n_e(e) = k$ is inconsistent 
for some $k\in \omega$. We may assume that $k$ is the least such number.
Then $\T \vdash \varphi^n_e(e) \neq k$ with some minimal proof $p$, so $f(e) = \varphi^n_e(e) = k$ by definition. With $\T$ extending $\ISn$, we have 
$\T + \mathrm{Th}_{\Pi_{n}}(\mathbb{N}) \vdash \varphi^n_e(e) = k$ 
by Fact \ref{fact:definability}.\ref{fact:definability_three}. But then
$\T + \mathrm{Th}_{\Pi_{n}}(\mathbb{N})$ is inconsistent, which by 
Fact \ref{fact:soundness} contradicts the assumption that $\T$ is $\Sigma_n$-sound. 
Hence the theory $\T + \varphi^n_e(e) = k$ is consistent for any choice of $k \in \omega$.

In this final part of the proof, we show that $\xi(x)$ is independent over $\T$.
Let $g$ be any function from $\omega$ to $\{0,1\}$ and let $X = \{ \xi(k)^{g(k)} : k\in \omega\}$.
Let $Y$ be any finite subset of $X$, and let $Z$ be the set $\{k: \xi(k) \in Y\}$. 
Let $\zeta(x) \defeq x \varepsilon a$, where $a$ is a code for the finite set $Z$; 
then $\zeta$ binumerates $Z$ in $\tq$. 
 
 Reason in the consistent theory $\T + \varphi^n_e(e) = \ulcorner \zeta\urcorner$:
 \begin{quote}
   If $\xi(x)$, then $\varphi^n_e(e) = z \land \mathrm{Sat}_{\Sigma_{n+1}}(z,x)$ 
   for some $z$. But $z$ is unique and $\varphi^n_e(e) = \ulcorner \zeta\urcorner$, 
   so $\mathrm{Sat}_{\Sigma_{n+1}}(\ulcorner \zeta\urcorner,x)$ and therefore $\zeta(x)$.
   
   Conversely, observe that since $\varphi^n_e(e) = \ulcorner \zeta\urcorner$, 
   $\xi(x)$ follows from $\zeta(x)$ by Fact \ref{fact:satisfaction}
   and $\exists$-introduction.
 \end{quote}

Hence the theory 
$\T + \forall x (\xi(x) \leftrightarrow \mathrm{Sat}_{\Sigma_{n+1}}(\ulcorner\zeta\urcorner,x))$ is consistent. If $k \in Z$, then $\T \vdash \mathrm{Sat}_{\Sigma_{n+1}}(\ulcorner\zeta\urcorner,k)$, so 
$\T + \forall x (\xi(x) \leftrightarrow \mathrm{Sat}_{\Sigma_{n+1}}(\ulcorner\zeta\urcorner,x)) \vdash \xi(k)$, and similarly for $k \notin Z$.
Hence the consistent theory 
$\T + \forall x (\xi(x) \leftrightarrow \mathrm{Sat}_{\Sigma_{n+1}}(\ulcorner\zeta\urcorner,x))$ 
proves all the sentences in $Y$, so $\T + Y$ is consistent. By compactness, it 
follows that $\T + X$ is consistent, and therefore $\xi(x)$ is independent over $\T$.
\end{proof}

The G\"odel-Rosser incompleteness theorem \ref{theorem:1} for arithmetically 
definable theories follows immediately from the result above. 
While the generalisation to arithmetically definable theories is new,
the basic idea of this proof is due to \citet[Corollary 1.1]{kripke1962}, who 
used it to rederive Mostowski's result from his own theorem on the existence of 
flexible formulae. Here, we understand flexibility in the following sense:

\begin{definition}
 A formula $\gamma(x)$ is \emph{flexible} for $\Gamma$ over $\T$ if, for every 
 $\delta(x) \in \Gamma$, the theory $\T + \forall x (\gamma(x) \leftrightarrow \delta(x))$ 
 is consistent.
\end{definition}

The definitions used by Kripke obscure the original content of his theorem, 
but, in hindsight, his proof yields that for every consistent, r.e.\ extension 
$\T$ of $\ID$, there is a $\Sigma_{n+1}$ formula that is flexible for 
$\Sigma_{n+1}$ over $\T$. Striving for some unification, we derive a hierarchical 
version of Kripke's theorem by generalising a result of Lindstr\"om's 
\citeyearpar[Proposition 2]{lindstrom1984}; which in turn is a generalisation 
of both Mostowski's and Kripke's results, as well as of Scott's famous lemma 
used to realise countable Scott sets as standard systems of models of 
$\pa$ \citeyearpar{scott1962}.

\begin{thm}\label{theorem:flexibility}
 Let $\T$ be a $\Sigma_{n+1}$-definable
 extension of $\mathrm{I}\Sigma_{m} + \mathrm{exp}$, with $m\geq n$. 
 For every $\Sigma_m$ formula $\phi(x)$, there is a $\Sigma_{m+1}$ 
 formula $\gamma(x)$ such that for every $g : \omega \to \{0,1\}$, if
 \[
  \T_g = \T + \{ \phi(k) ^{g(k)}:k\in\omega\}
 \]
is $\Sigma_n$-sound, then $\gamma(x)$ is flexible for $\Sigma_{m+1}$ over $\T_g$.
\end{thm}

\begin{proof}
 Fix $n$ and let $\phi(x) \in \Sigma_m$, with $m\geq n$. Let 
 $f(s,\ulcorner \eta \urcorner) = \ulcorner \sigma \urcorner$ iff the following holds: 
 \begin{enumerate}
  \item $s$ is binary sequence of length $k + 1$;
  \item there is a proof $p$ of $\lnot \forall x (\eta(x) \leftrightarrow \sigma(x))$ 
        in $\T + \phi(0)^{(s)_0} + \dots + \phi(k)^{(s)_k}$;
  \item for every $q<p$ and any $\ulcorner\sigma_0\urcorner \leq \ulcorner \sigma \urcorner$, 
        $q$ is not a proof of $\lnot \forall x (\eta(x) \leftrightarrow \sigma_0(x))$ 
        in $\T + \phi(0)^{(s)_0} + \dots + \phi(k)^{(s)_k}$.
 \end{enumerate}
Here $(s)_k$ denotes the $k$th element of the sequence $s$.

Since $\T$ is $\Sigma_{n+1}$-definable, there is a deductively equivalent 
$\Pi_n$-definition of $\T$. Hence, the relation $f(x,y) = z$ is r.e.\ in 
$\emptyset^{(n)}$, and therefore partial $n$-recursive. Let $e$ be an $n$-index for $f$.

Let $\mathrm{Seq}_\phi(x)$ be the formula
\[
 \forall y < l(x)(y \varepsilon x \leftrightarrow \phi(y))
\]
where $l(x)$ denotes the length of $x$ (this is the formula ``$x$ is a $l(x)$-piece of $\phi$''). 
Whenever $\phi(x)$ is $\Sigma_m$, $\mathrm{Seq}_\phi(x)$ is $\Sigma_0(\Sigma_m)$, and 
since $\T \vdash \mathrm{I}\Sigma_m$, it is $\Delta_{m+1}$ in $\T$. 
Let, by Fact \ref{fact:diagonal}, $\gamma(x)$ be such that
\[
\T \vdash \forall x(\gamma(x) \leftrightarrow \exists s \exists z(\mathrm{Seq}_\phi(s) \land \varphi^n_e(s, \ulcorner \gamma \urcorner)= z \land \mathrm{Sat}_{\Sigma_{m+1}}(z,x)))\text{.}
\]

Since $\T \vdash \mathrm{I}\Sigma_m$ and $m\geq n$, the formula strongly 
representing $f$ in $\T + \mathrm{Th}_{\Pi_{n}}(\mathbb{N})$ is equivalent 
to a $\Sigma_{n+1}$ formula in $\T$. It follows that $\gamma(x)$ is equivalent 
to a $\Sigma_{m+1}$ formula in $\T$.

Suppose, for a contradiction, that there is a $g : \omega \to \{0,1\}$ and a 
$\sigma(x) \in \Sigma_{m+1}$ such that $\T_g$ is $\Sigma_n$-sound, but 
$\T_g + \forall x(\gamma(x) \leftrightarrow \sigma(x))$ is inconsistent.
If there are more than one such $\sigma(x)$ for a given $g$, consider the 
one with the least Gödel number. There is then an initial subsequence $s$ of 
$g$, of length $k+1$ for some $k$, such that $p$ is a proof of 
$\lnot \forall x(\gamma(x) \leftrightarrow \sigma(x))$ in 
$\T + \phi(0)^{(s)_0} + \dots + \phi(k)^{(s)_k}$. Let $s$ be the initial 
subsequence of $g$ corresponding to the least such $p$.

It is now clear that neither $p$ nor any $q <p$ can be a proof of 
$\lnot \forall x(\gamma(x) \leftrightarrow \sigma_0(x))$ in 
$\T + \phi(0)^{(s)_0} + \dots + \phi(k)^{(s)_k}$ for any formula $\sigma_0(x)$ 
whose Gödel number is less than $ \ulcorner \sigma \urcorner$. Hence, 
by definition, $f(s,\ulcorner \gamma \urcorner) = \ulcorner \sigma \urcorner$, and by Fact \ref{fact:definability},
\[
\T + \Th_{\Pi_{n}}(\mathbb{N}) \vdash \varphi^n_e(s, \ulcorner \gamma \urcorner) = \ulcorner \sigma\urcorner\text{.}
\] 
By choice of $s$, $\T_g \vdash \mathrm{Seq}_\phi(\ulcorner s \urcorner)$, so 
$\T_g + \Th_{\Pi_{n}}(\mathbb{N}) \vdash \forall x (\gamma(x) \leftrightarrow \sigma(x))$ 
by an argument similar to that in the proof of Theorem \ref{theorem:independence}. 
Then $\T_g +\Th_{\Pi_{n}}(\mathbb{N})$ is inconsistent, contradicting the assumption 
that $\T_g$ was $\Sigma_n$-sound.
\end{proof}

By choosing $\phi(x)$ as $\top$ in the construction above, we obtain the expected 
hierarchical version of Kripke's theorem. A similar, but not entirely correct, claim
is made by \citet[Theorem 4.8]{blanck2017a}.
 \begin{cor}\label{cor:kripke}
 Let $\T$ be a $\Sigma_{n+1}$-definable, $\Sigma_n$-sound extension of 
 $\ISn + \mathrm{exp}$. For all $m\geq n$, there is a $\Sigma_{m+1}$ formula 
 $\gamma(x)$ that is flexible for $\Sigma_{m+1}$ over $\T$.
\end{cor}
Mostowski's theorem for r.e.\ extensions of $\ID$ then follows immediately by 
using the method described in the proof of Theorem \ref{theorem:independence}. 
A similar argument also yields Scott's lemma. 

The next objective is to show how the hierarchical version of Kripke's theorem 
can be formalised in  $\Pi_n$-complete extensions of $\ISnp$. 
A similar, but not entirely correct, claim is made by \citet[Theorem 5.8]{blanck2017a}. The present 
proof is a minor modification of an argument of \citet[Theorem 5.1]{blanck2017a}.
  
 \begin{thm}\label{theorem:formal_flexibility}
  Let $\s$ be a $\Pi_n$-complete, consistent extension
  of $\ISnp$, and let $\T$ be $\Sigma_{n+1}$-definable. For all $m \geq n$, 
  there is a $\Sigma_{m+1}$ formula $\gamma(x)$ such that:
  \begin{enumerate}
   \item $\ISnp \vdash \con_{\T,\Sigma_{n+1}} \rightarrow \forall x \lnot \gamma(x)$;
   \item if $\sigma(x) \in \Sigma_{m+1}$, then every model of 
        $\s + \con_{\T,\Sigma_{n+1}}$ has a $\Sigma_n$-elementary extension 
        to a model of $\T + \forall x (\gamma(x) \leftrightarrow \sigma(x))$.
  \end{enumerate}
 \end{thm}

 \begin{proof}
  Fix $n$, and let $\phi(x,z)$ be the $\Sigma_{n+1}$ formula 
  $\pr_{\T,\Sigma_{n+1}}(\ulcorner \varphi^n_{\dot{x}}(\dot{x}) \neq \dot{z} \urcorner)$. 
  Let $e$ be the G\"odel number of $\phi(x,z)$. 
  
  Recall that $\varphi^n_x(y) = z$ is shorthand for 
  $\mathrm{Sel}\{\mathrm{Sat}_{\Sigma_{n+1}}\}(x,y,z)$. Therefore, by Fact \ref{fact:selection}.1 we have
    \begin{equation}\label{eq:1}
   \ISnp \vdash \forall z (\varphi^n_{e}(e) = z \rightarrow \mathrm{Sat}_{\Sigma_{n+1}}(e,e,z))\text{,}
  \end{equation}
  so by construction of $\phi(x,z)$ and choice of $e$,
  \begin{equation}\label{eq:2}
   \ISnp \vdash \forall z (\varphi^n_{e}(e) = z \rightarrow \pr_{\T,\Sigma_{n+1}}(\ulcorner \varphi^n_{e}(e) \neq \dot{z}\urcorner))\text{.}
  \end{equation}
By Fact \ref{fact:completeness}
  \begin{equation}\label{eq:3}
   \ISnp \vdash \forall z (\varphi^n_{e}(e) = z \rightarrow \pr_{\T,\Sigma_{n+1}}(\ulcorner \varphi^n_{e}(e) = \dot{z}\urcorner))\text{,}
  \end{equation}
  so (\ref{eq:2}) and (\ref{eq:3}) together with Fact \ref{fact:lob} give
  \begin{equation}\label{eq:4}
   \ISnp \vdash \forall z (\con_{\T,\Sigma_{n+1}} \rightarrow \varphi^n_{e}(e) \neq z)\text{.}
  \end{equation}

  Now, observe that
  \begin{equation}\label{eq:5}
   \ISnp \vdash \exists z \lnot \con_{\T + \varphi^n_{e}(e) = \dot{z},\Sigma_{n+1}} \leftrightarrow \exists z \pr_{\T,\Sigma_{n+1}}(\ulcorner \varphi^n_{e}(e) \neq \dot{z}\urcorner)\text{.}
  \end{equation}
By construction of $\phi(x,z)$, the right hand side of the equivalence of 
(\ref{eq:5}) is identical to $\exists z \phi(e,z)$, and
by Fact \ref{fact:satisfaction}, we have
\begin{equation}\label{eq:6}
\ISnp \vdash \exists z \phi(e,z) \leftrightarrow \exists z \mathrm{Sat}_{\Sigma_{n+1}}(e,e,z)\text{.}
\end{equation}
Therefore, by Fact \ref{fact:selection}.3 and the convention on $\varphi^n_{e}$, (\ref{eq:5}) gives
\begin{equation}\label{eq:7}
   \ISnp \vdash \exists z \lnot \con_{\T + \varphi^n_{e}(e) = \dot{z},\Sigma_{n+1}} \leftrightarrow \exists z (\varphi^n_{e}(e) = z)\text{.}
\end{equation}

 Together with (\ref{eq:4}), this implies
 \begin{equation}\label{eq:8}
  \ISnp \vdash \forall z (\con_{\T,\Sigma_{n+1}} \rightarrow \con_{\T+\varphi^n_{e}(e) = \dot{z},\Sigma_{n+1}})\text{.}
 \end{equation}

Then the $\Sigma_{m+1}$ formula 
  $\gamma(x) \defeq \exists z (\varphi^n_{e}(e) = z \land \Sat_{\Sigma_{m+1}}(z,x))$ 
  is as desired, and the first part of the theorem follows directly from (\ref{eq:4}).
    For the second part, let $\mathcal{M}$ be any model of $\s + \con_{\T,\Sigma_{n+1}}$, 
    and let $\sigma(x)$ be any $\Sigma_{m+1}$ formula. By (\ref{eq:8}), we immediately 
    get $\mathcal{M} \models \con_{\T+\varphi^n_e(e) = \ulcorner\sigma\urcorner,\Sigma_{n+1}}$. 
    Moreover, since $\s$ is a $\Pi_n$-complete extension of $\ISnp$, and $\T$ is $\Sigma_{n+1}$-definable, the theory 
    $\T + \varphi^n_e(e) = \ulcorner\sigma\urcorner$ is $\Sigma_{n+1}$-definable in 
    $\mathcal{M}$, using Craig's trick.
    By Fact \ref{fact:ACT}, there is then a $\Sigma_{n}$-elementary end-extension 
    $\mathcal{K}$ of $\mathcal{M}$ that satisfies $\T + \varphi^n_e(e) = \ulcorner\sigma\urcorner$.
  Since $\mathcal{K}$ satisfies $\varphi^n_e(e) =\ulcorner \sigma \urcorner$, 
  it follows that $\mathcal{K} \models \forall x (\gamma(x) \leftrightarrow \sigma(x))$, as desired.
 \end{proof}

 \begin{cor}
 Let $\s$ be a $\Sigma_{n+1}$-definable, $\Pi_n$-complete, and consistent extension of $\ISnp$, and let $\gamma(x)$ be as in the proof of Theorem \ref{theorem:formal_flexibility}.
 Then, for each $\sigma(x) \in \Sigma_{m+1}$ with $m\geq n$,
$\s + \forall x (\gamma(x) \leftrightarrow \sigma(x))$ is $\Pi_{n+1}$-conservative 
over $\s + \con_{\s,\Sigma_{n+1}}$.
\end{cor}

  The next theorem has a different flavour than the earlier ones, and is a 
  generalisation of Woodin's theorem on the universal algorithm \citeyearpar{woodin2011}; 
  see also \citet[Theorem 3.1]{blanck_enayat2017}. A version for r.e.\ extensions of 
  $\pa$ is independently due to \citet[Theorem 18]{hamkins_20XXa}, and the proof 
  presented here uses a method that I learned from Shavrukov.
  The particular Solovay-style construction used in the proof is similar to the ones used by 
  \cite{berarducci1990} and \cite{japaridze1994}.
 \begin{thm}\label{theorem:woodin}
  Let $\T$ be a $\Sigma_{n+1}$-definable, $\Pi_n$-complete, and consistent extension of $\ISnp$. 
  There is a $\Sigma_{n+1}$-definable set $W_e$ such that:
  \begin{enumerate}
   \item $\ISnp + \mathrm{Th}_{\Pi_n}(\mathbb{N}) \vdash \text{``}W_e\text{ is finite''}$;
   \item $\ISnp + \mathrm{Th}_{\Pi_n}(\mathbb{N}) \vdash \con_{\T,\Sigma_{n+1}} \rightarrow W_e = \emptyset$;
   \item for each countable model $\mathcal{M} \models \T$, if $s$ is an 
        $\mathcal{M}$-finite set such that $\mathcal{M} \models W_e \subseteq s$, 
        then there is a $\Sigma_n$-elementary extension of $\mathcal{M}$ satisfying 
        $\T + W_e = s$.
  \end{enumerate}
 \end{thm}

\begin{proof}
The set $W_e$ is defined (in $\ISnp + \mathrm{Th}_{\Pi_n}(\mathbb{N})$ and in the real world) as follows, using the formalisation of the recursion theorem (Fact \ref{fact:recursion}).
At the same time, an auxiliary function $r(x)$ is defined.
 
 \textbf{Stage $0$}: Set $W_{e,0} = \emptyset$, and $r(0) = \infty$.\footnote{Here 
 $\infty$ is a formal symbol that by definition is greater than all the natural numbers.}

 \textbf{Stage $x+1$}: Suppose $r(x) = m$. There are two cases:
\begin{quote}
\textbf{Case A}: $s$ is a finite set such that $s \supseteq W_{e,x}$,  $k < m$,
and $x$ witnesses a $\Sigma_{n+1}$ 
 formula $\sigma(s)$ such that $k$ is a proof in $\T + \Th_{\Sigma_{n+1}}(\mathbb{N})$ 
 of $\forall t (\sigma(t) \rightarrow W_e \neq t)$. 
 Should there be more than one eligible candidate for either $k$ or $s$, then choose the least such $k$, and then the least $s$ corresponding to that $k$.
 Then set $W_{e,x+1} = s$ and $r(x+1) = k$.

 \noindent\textbf{Case B}: Otherwise, set $W_{e,x+1} = W_{e,x}$ and $r(x+1) = m$.
\end{quote}
Let $W_e = \bigcup_x W_{e,x}$. 

To prove 1., reason as follows:
Since the proof relation for $\T + \Th_{\Sigma_{n+1}}(\mathbb{N})$ is $\Delta_{n+1}$ in $\ISnp + \mathrm{Th}_{\Pi_n}(\mathbb{N})$,
$W_e$ is r.e.\ in $\emptyset^{(n)}$, and therefore $\Sigma_{n+1}$ by Fact \ref{fact:post}.
Provably in $\ISnp + \mathrm{Th}_{\Pi_n}(\mathbb{N})$, we have that $W_{e,x+1} \supseteq W_{e,x}$, and $r(x+1) \leq r(x)$, 
so by the $\Sigma_{n+1}$-least number principle (which is available thanks to $\Sigma_{n+1}$-induction), there is a limit $R = \lim_x r(x)$.
For each $x$ with $W_{e,x+1} \neq W_{e,x}$, $\ISnp+ \mathrm{Th}_{\Pi_n}(\mathbb{N})$ proves $r(x+1) < r(x)$, whence 
there can only be finitely many such $x$. So $\ISnp + \mathrm{Th}_{\Pi_n}(\mathbb{N})\vdash \text{``}W_e \text{ is finite''}$.

Note also that $\T \vdash R > k$ for all $k \in \omega$. To show this, fix $k \in \omega$ 
and argue in $\T$:
\begin{quote}
 Suppose $R \leq k$. Let $y$ be minimal such that $r(y+1) = R$. Then 
 $W_e = W_{e,y+1} = s$ for some $s$ such that $R$ is a proof in 
 $\T + \Th_{\Sigma_{n+1}}(\mathbb{N})$ of $\forall t (\sigma(t)  \rightarrow W_e \neq t)$, 
 where $\sigma(s)$ is a true $\Sigma_{n+1}$ formula. 
\end{quote}
But, by Fact \ref{fact:exp_reflection},
\begin{quote}
since $\forall t (\sigma(t) \rightarrow W_e \neq t)$ is proved from a true 
$\Sigma_{n+1}$ sentence with a proof not exceeding $k$, it must be true. 
Since $\sigma(s)$ is true, $W_e \neq s$ is also true, and the contradiction 
proves $R > k$.
\end{quote}
To prove 2., argue for the contrapositive statement in $\ISnp +\Th_{\Pi_{n}}(\mathbb{N})$:
\begin{quote}
If $W_e = s \neq \emptyset$, then 
$\pr_{\T,\Sigma_{n+1}}^m(\ulcorner\forall t (\sigma(t) \rightarrow W_e \neq t)\urcorner)$ 
for some $m$ and some true $\Sigma_{n+1}$ formula $\sigma(s)$.
The relation $s \subseteq W_{e,x}$ is $\Delta_{n+1}$ by Fact \ref{fact:closure}.3, so
 $\pr_{\T,\Sigma_{n+1}}(\ulcorner \dot{s} \subseteq W_e \urcorner)$ follows by 
Fact \ref{fact:completeness}. 
Now reason inside $\pr_{\T,\Sigma_{n+1}}$: 

\begin{minipage}{0.85\textwidth}
\begin{quote}
 There is some $u = W_e$ with $u \supseteq s$, so by construction, $\sigma'(u)$ is 
 true, and $\pr_{\T,\Sigma_{n+1}}^k(\ulcorner \forall t (\sigma'(t) \rightarrow W_e \neq t)\urcorner)$ for some $k\leq m$ and some $\Sigma_{n+1}$ formula $\sigma'(u)$. 
 \end{quote}
\end{minipage}

\noindent Apply Fact \ref{fact:reflection}, and continue reasoning inside $\pr_{\T,\Sigma_{n+1}}$:

\begin{minipage}{0.85\textwidth}
\begin{quote}
Then $\forall t (\sigma'(t) \rightarrow W_e \neq t)$ and $\sigma'(u)$, so $W_e \neq u$.
\end{quote}
\end{minipage}

\noindent Then $\pr_{\T,\Sigma_{n+1}}(\ulcorner \exists u(W_e = u \land W_e \neq u)\urcorner)$, 
so Fact \ref{fact:lob} gives $\lnot \con_{\T,\Sigma_{n+1}}$ as desired.
\end{quote}
 
To prove 3., first fix $m \in \omega$. By Fact \ref{fact:exp_reflection}, there is 
a proof $k$ in $\T$ of 
\[
\forall t (\pr^m_{\T, \Sigma_{n+1}}(\ulcorner W_e \neq \dot{t} \urcorner) \rightarrow W_e \neq t)\text{.}
\]
Now reason in $\T$:
\begin{quote}
 Consider any finite $s \supseteq W_e$, and suppose $x$ is a proof $\leq m$ 
 of $W_e \neq s$ in $\T + \Th_{\Sigma_{n+1}}(\mathbb{N})$. 
 Then $s \supseteq W_{e,x+1}$, and therefore $r(x+1) \leq k$ by construction of 
 $r(x+1)$: here $\pr^m_{\T, \Sigma_{n+1}}(\ulcorner W_e \neq \dot{s} \urcorner)$ is 
 $\ISnp$-equivalent to a true $\Sigma_{n+1}$ sentence playing the role of $\sigma(s)$. But 
 $k < R \leq r(x+1)$, and the contradiction proves $\con^m_{\T + W_e = \dot{s}, \Sigma_{n+1}}$.
 \end{quote}
 
 Therefore for all $m\in\omega$, $\T \vdash \forall s \supseteq W_e \,\, \con^m_{\T + W_e=\dot{s}, \Sigma_{n+1}}$.
 
For the final part of the proof, let $\mathcal{M}$ be any countable model of $\T$, and let $s$ be any $\mathcal{M}$-finite set such that $\mathcal{M} \models W_e \subseteq s$. 
Since $\T$ is a $\Sigma_{n+1}$-definable, $\Pi_{n}$-complete extension of $\ISnp$,
Facts \ref{fact:satisfaction} and \ref{fact:code} imply that
$\T + \Th_{\Sigma_{n+1}}(\mathcal{M}) + W_e = s \in \SSy(\mathcal{M})$. Since 
$\T \vdash \con^m_{\T + W_e = \dot{s}, \Sigma_{n+1}}$ for all $m\in \omega$, 
Fact \ref{fact:OHGL} guarantees the existence of a $\Sigma_n$-elementary extension 
of $\mathcal{M}$ satisfying $\T + W_e = s$, which concludes the proof of the theorem.
\end{proof}

\begin{cor}\label{cor:kreisel}
With $\T$ as in Theorem \ref{theorem:woodin}, $\lnot \con_{\T,\Sigma_{n+1}}$ is 
$\Pi_{n+1}$-conservative over $\T$.
\end{cor}

\begin{proof}
 Every countable model of $\T$ has a $\Sigma_n$-elementary extension 
 satisfying $W_e \neq \emptyset$, and therefore also 
 $\T + \lnot \con_{\T,\Sigma_{n+1}}$ by the Theorem. By Fact \ref{fact:OHGL}, the conclusion follows.\footnote{As pointed out by one of the referees, adapting Kreisel's original proof of 
the $\Pi_1$-conservativity of $\lnot \con_\T$ over $\T$ is a simpler way to 
establish Corollary \ref{cor:kreisel} than going via Theorem \ref{theorem:woodin}.
}
\end{proof}

The set $W_e$ defined in Theorem \ref{theorem:woodin} can be used to prove 
results of a more Kripkean variety, by using the information contained in 
$W_e$ as codes for other sets. The next result
is of this kind, and improves on Theorem 7.21 of \citet{blanck2017a} by 
generalising to arithmetically definable theories. A version for r.e.\ extensions 
of $\pa$ is independently due to Hamkins \citeyearpar[Theorem 22(1)]{hamkins_20XXa}, who 
also noted that there is a very short proof of it from Theorem \ref{theorem:woodin}.
 
\begin{thm}\label{theorem:uniform_flexibility}
 Let $\T$ be a $\Sigma_{n+1}$-definable, $\Pi_n$-complete, and consistent extension of $\ISnp$. 
 For all $m \geq n$, there is a $\Sigma_{m+2}$ formula $\gamma(x)$ such that:
 \begin{enumerate}
   \item $\ISnp + \mathrm{Th}_{\Pi_n}(\mathbb{N}) \vdash \con_{\T,\Sigma_{n+1}} \rightarrow \forall x  \lnot \gamma(x)$;
  \item for every $\sigma(x) \in \Sigma_{m+2}$, every countable model of $\T$ 
  has a $\Sigma_n$-elementary extension satisfying 
  $\T + \forall x(\gamma(x) \leftrightarrow \sigma(x))$.
 \end{enumerate}
\end{thm}

\begin{proof}[Proof sketch]
It is straightforward to adapt the construction in the proof of Theorem \ref{theorem:woodin} 
to produce an $\mathcal{M}$-finite binary sequence $S_e$, rather than a set 
\citep{woodin2011, blanck_enayat2017,hamkins_20XXa}. Assume an enumeration of 
$\Sigma_{m+2}$ formulae in which every $\Sigma_{m+2}$ formula occurs infinitely 
often, and that every finite binary sequence codes such a formula.
Let $\gamma(x)$ be the formula $\exists z (S_e = z \land \Sat_{\Sigma_{m+2}}(z,x))$. 
Since $S_e = z$ is at most $\Sigma_{n+2}$ and $m \geq n$, it follows that $\gamma(x)$ is $\Sigma_{m+2}$.

Pick any $\sigma(x) \in \Sigma_{m+2}$, let $\mathcal{M}$ be any countable model 
of $\T$ and let $s$ be $S_e$ as calculated within $\mathcal{M}$. By assumption 
on the enumeration of $\Sigma_{m+2}$ formulae, there is an $\mathcal{M}$-finite 
sequence $t \supseteq s$ such that $t$ codes $\ulcorner \sigma(x) \urcorner$.
By the sequence version of Theorem \ref{theorem:woodin}, there is a $\Sigma_n$-elementary 
extension $\mathcal{K}$ of $\mathcal{M}$ in which $S_e =t$. Then $\gamma(x)$ coincides 
with $\sigma(x)$ in $\mathcal{K}$, and therefore is as desired.
\end{proof}

The question remains to which extent $m+2$ can be replaced by $m+1$ in the statement of Theorem \ref{theorem:uniform_flexibility}. Some partial answers are already available: Theorem \ref{theorem:woodin} gives a positive answer restricted to $\Sigma_{m+1}$ formulae $\sigma(x)$ whose extension is $\mathcal{M}$-finite and for which $\mathcal{M} \models \forall x (\gamma(x) \rightarrow \sigma(x))$, while
Theorem \ref{theorem:formal_flexibility} can be seen as giving a partial positive answer that is restricted to models of $\T + \con_{\T,\Sigma_{n+1}}$. 
\citet[Chapter 7.4]{blanck2017a} lists several other partial answers to this question in a setting where $\T$ is an r.e.\ extension of $\pa$ and $n=0$. By using the principles of Section \ref{section:principles} of the present paper, those constructions can be easily modified to give equally unsatisfactory answers in the present setting.
The salient remaining question is as follows:

\begin{quest}
{ Let $\T$ be a $\Sigma_{n+1}$-definable, $\Pi_n$-complete, and consistent extension of $\ISnp$. 
 Is there a $\Sigma_{n+1}$ formula $\gamma(x)$ such that:
 \begin{enumerate}
   \item $\ISnp + \mathrm{Th}_{\Pi_n}(\mathbb{N}) \vdash \con_{\T,\Sigma_{n+1}} \rightarrow \forall x  \lnot \gamma(x)$;
  \item for every $\sigma(x) \in \Sigma_{n+1}$, every countable model of $\T + \forall x(\gamma(x) \rightarrow \sigma(x))$ 
  has a $\Sigma_n$-elementary extension satisfying 
  $\T + \forall x(\gamma(x) \leftrightarrow \sigma(x))$?
 \end{enumerate}}
\end{quest}

\section{Discussion}
By inspecting the results proved in Section \ref{section:applications}, we see two classes of $\Sigma_{n+1}$-definable theories emerging:
\begin{enumerate}
 \item $\Sigma_n$-sound extensions of $\ISn + \mathrm{exp}$; and
 \item $\Pi_n$-complete, consistent extensions of $\ISnp$.
\end{enumerate}
As suggested by Theorem \ref{theorem:1}, theories in the first class are strong enough for some applications. These include the results of \citet{salehi_seraji2017} and \citet{kikuchi_kurahashi2017}, together with Theorems \ref{theorem:independence} and 
\ref{theorem:flexibility} (and their corollaries) of the present paper. This success relies 
on the fact that $\Sigma_n$-soundness of $\T$ guarantees the consistency of 
$\T + \mathrm{Th}_{\Pi_{n}}(\mathbb{N})$, in which the $n$-recursive functions 
can be strongly represented by a formula that is $\Sigma_{n+1}$ in the presence of $\Sigma_n$-induction. 

The second class of theories is required to prove results on $\Sigma_n$-elementary
extensions of models of $\Sigma_{n+1}$-definable theories,
for example results on partial conservativity via the OHGL characterisation (Theorem \ref{theorem:formal_flexibility} and onwards). In these cases, $\Pi_n$-completeness
of $\T$ ensures that every model $\mathcal{M}$ of $\T$ is a $\Sigma_n$-elementary extension of the standard model, which in the presence of $\Sigma_{n+1}$-induction suffices for $\T$ to be $\Sigma_{n+1}$-definable in $\mathcal{M}$ by using Craig's trick. $\Sigma_{n+1}$-induction is also used to prove the 
arithmetised completeness theorem for $\Sigma_{n+1}$-definable theories, which is indispensable for constructing the $\Sigma_n$-elementary extensions.

The Facts listed in Section 3 should be enough to derive hierarchical generalisations 
for arithmetically definable extensions of fragments of $\pa$ of many of the theorems in, 
e.g., Lindstr\"om's classic \emph{Aspects of Incompleteness} \citeyearpar{lindstrom2003}. 
As suggested by the results in the present paper, some of these generalisations would 
apply only to $\Pi_n$-complete theories, while in other cases mere $\Sigma_n$-soundness would do. 
Others might not be prone to such generalisations at all, as shown by 
\citet[Theorem 11]{kurahashi2018} and pointed out to me by one of the referees. 
In fact, it would be interesting to see which of the results in, say, the first 
5 chapters of \emph{Aspects} (where the results do not depend on $\T$ being 
essentially reflexive) that are prone to such generalisations, using these principles.

\section{Acknowledgements}
This paper is based in part on some ideas that were left half-baked in 
the author's doctoral thesis \citet{blanck2017a}, written under the supervision 
of Ali Enayat. In particular, not entirely correct claims similar to 
Corollary \ref{cor:kripke} and Theorem \ref{theorem:formal_flexibility} 
have appeared there. I am grateful to Ali Enayat, Fredrik Engstr\"om, 
Joel David Hamkins, Aleksandre Maskharashvili, and Volodya Shavrukov for 
inspiration, discussion, and guidance. The comments of two excellent anonymous
referees have greatly helped improve the paper by suggesting the 
correct statements of some of the theorems, and by weeding out a number of 
false claims in an earlier version of this paper.

\vspace*{10pt}

\clearpage

\end{document}